\documentclass[11pt]{amsart}
\usepackage{amsmath}
\usepackage{amssymb}
\usepackage{amsfonts}
\usepackage{mathrsfs}
\newtheorem{theorem}{Theorem}[section]

\newtheorem{lemma}[theorem]{Lemma}

\theoremstyle{remark}

\newcommand{\uno}{\textbf{1}}

\def\QQ{{\mathbb{Q}}}
\def\CC{{\mathbb{C}}}

\def\lx{\mathcal{B}(X)}
\def\ly{\mathcal{B}(Y)}
\def\lxy{\mathcal{B}(X,Y)}
\def\lyx{\mathcal{B}(Y,X)}
\def\lh{\mathcal{B}(H)}

\def\ker{{\rm Ker}}

\def\m{{\rm m}}
\def\q{{\rm q}}
\def\M{{\rm M}}
\def\c{{\rm c}}

\def\ran{{\rm ran}}
\def\dist{\rm dist}

\def\r{{\rm r}}

\title[The reduced minimum modulus additive preservers]{Additive maps preserving the reduced minimum
 modulus of Banach\\
  space operators}
\author{Abdellatif Bourhim}

\address{D\'epartement de math\'ematiques et de statistique, Universit\'e Laval, Qu\'ebec, Canada}
\address{\textit{Current address}: Department of Mathematics, Syracuse University, 215 Carnegie Building,
Syracuse, New York 13244}
\email{bourhim@mat.ulaval.ca \& abourhim@syr.edu}



\thanks{The author was supported by an adjunct professorship at Laval university.}

\begin{document}
\begin{abstract}Let  $\lx$ be the algebra of all bounded linear operators on an infinite dimensional complex Banach space $X$. We prove that 
an additive surjective map $\varphi$ on $\lx$ preserves the reduced minimum modulus  if and only if either
there are bijective isometries $U:X\to X$ and $V:X\to X$ both linear or both conjugate linear such that $\varphi(T)=UTV$ for all 
$T\in\lx$, or  $X$ is reflexive and there are bijective isometries $U:X^*\to X$ and $V:X\to X^*$ both linear or both conjugate linear such that 
$\varphi(T)=UT^*V$ for all $T\in\lx$. As immediate consequences of the ingredients used in the proof of this result, we get the complete description 
of surjective additive maps preserving  the minimum, the surjectivity and the maximum moduli of Banach space operators.  
\end{abstract}
\maketitle

\vspace*{-0.7cm}

\section{Introduction}

Several results on linear preservers have been extended to the setting of additive preservers, and, in many cases, their extensions
demonstrated to be nontrivial as the forms of additive preservers are some time not  ``nice" as 
the ones of the corresponding linear preservers. In \cite{OSP}, Omladi\v c and \v Semrl characterized surjective additive maps preserving the spectrum
of bounded linear operators on complex Banach spaces and showed that such maps are of standard forms. This is an extension of the result due to 
Jafarian and Sourour \cite{JaSo86} that describes linear spectrum-preserving maps. 
In \cite{ZFZH},  Bai and Hou considered a more general situation  and characterized surjective additive maps preserving the spectral radius of Banach space operators, extending the result due to Bre\v sar and \v Semrl \cite{BreSem96} from the linear setting. For further results on additive preserver problems, we refer the interested reader, for example, to \cite{ZFZH, bai-hou, bourhim, cui, hoi-Meiyan, hou, Cui-Hou, OS, OSP, zhao-hou} and the references therein.

Recently, Mbekhta described unital surjective linear maps on $\lh$, the algebra  of all bounded linear operators on an infinite dimensional 
complex Hilbert space $H$, preserving several spectral quantities such as the minimum, the surjectivity and the reduced minimum moduli; see \cite{Mb, Mb3}.
In \cite{Mb3}, he showed that a unital surjective linear map on $\lh$ preserves the reduced minimum modulus if and only if it is an isometry and conjectured that the 
same result remains true for the nonunital linear case. Mbekhta's articles \cite{Mb} and \cite{Mb3}, which were followed quickly by several papers treating related problems,  contain several good ideas and results which 
opened the way for certain authors to consider more general situations. His results were extended to a more general setting by characterizing (not necessarily unital) surjective linear maps between $C^*$-algebras
preserving the minimum, surjectivity, maximum, and reduced minimum moduli and his conjecture was positively settled; see \cite{BBS}. 
In \cite{skhiri}, Skhiri generalized Mbekhta's result by characterizing surjective linear maps on $\lx$, the algebra  of all bounded linear operators on an infinite dimensional complex Banach
space $X$, preserving the reduced minimum modulus. As the main result of \cite{skhiri}, he established the following theorem.
\begin{theorem}
A surjective linear map $\varphi$ from $\lx$ onto itself for which $\varphi(\uno)$ is invertible preserves the
reduced minimum modulus if and only if it is either an isometric automorphism or isometric antiautomorphism multiplied by a bijective isometry in $\lx$.
\end{theorem}
This result has been also proved in \cite[Theorem 7.1]{BBS} for the Hilbert space operators case but without the extra condition that $\varphi(\uno)$
is invertible. In fact, much more has been established in \cite{BBS} where it is shown that a surjective linear map $\varphi$ between $C^*$-algebras 
preserves the reduced minimum modulus if and only if it is a selfadjoint Jordan isomorphism multiplied by a unitary element. This result clearly shows that
the condition that $\varphi(\uno)$ is invertible in the above theorem is superfluous even for the more general setting of the reduced minimum modulus preservers 
between $C^*$-algebras. 

In this paper, we completely describe additive surjective maps preserving the reduced minimum modulus of Banach space operators. The obtained
result, which extends \cite[Theorem 7.1]{BBS}, improves Skhiri's result and shows that the condition that  $\varphi(\uno)$ is invertible in the above theorem is also superfluous even for the Banach space operators case. Our proof
is simple and  self-contained and also works to recapture and extend, to Banach space operators case, the recent results from \cite{bendaoud, BB1, Mb}
which describe linear and additive maps preserving the minimum modulus, the surjectivity modulus, and the maximum modulus of Hilbert space operators.
Unlike in \cite{bendaoud}, we avoid using several deep results such as Herstein theorem's \cite{Her56} and the celebrate Theorem of Kadison \cite{Ka51}.

\textit{Acknowledgements.} Some arguments of certain proofs presented herein are influenced by several ideas which the author 
shared with his colleagues M. Burgos and V.S. Shulman during the preparation of their joint paper \cite{BBS}. He wish to express his thanks 
to them for the enjoyable collaboration and for their remarks and comments.

\section{Main results}
Throughout this paper, $X$ and $Y$ denote infinite dimensional
complex Banach spaces, and $\lxy$ denotes the space of all 
bounded linear maps from $X$ into $Y$. As usual, when $X=Y$, we simply write $\lx$ instead of ${\mathcal B}(X,X)$. The \emph{reduced minimum modulus} of a map $T\in\lxy$
is defined by 
$$\gamma(T):=\left\{
\begin{array}{ll}
\inf\{\| Tx\| \colon {\dist}(x, \ker(T))\geq 1\}&\text{if }T\not=0\\
\infty&\text{if }T=0.
\end{array}
\right.$$
The reduced minimum modulus measures the closedness of the range of operators in the sense that  $\gamma(T)$ is positive precisely when $T$ has a closed range; see for instance \cite[II.10]{Mu}. 
Recall also that the \emph{minimum modulus} and the \emph{surjectivity modulus} of $T$  are defined respectively by 
$$\m(T):=\inf \{\, \|Tx\|\colon x\in
X,~\|x\|=1\, \}\text{ and }{\q}(T):=\sup\{\, \varepsilon\geq 0 \colon \varepsilon B_Y \subseteq T(B_X)\,\},$$  where  $B_X$ denotes the closed unit ball of $X$. 
Note that $\m(T)>0$ if and only if $T$ is injective and has closed range, and
that $\q(T)>0$ if and only if $T$ is surjective. While, the \emph{ maximum modulus} of $T$ is defined by $\M(T):=\max(\m(T),\q(T))$. It is easy to see that $\M(T)\leq \gamma(T)$ and that 
$\M(T)=\M(T^*)=\gamma(T)=\gamma(T^*)$ provided that $\M(T)>0$, where $T^*:Y^*\to X^*$ is the adjoint of $T$ acting between the dual spaces of $Y$ and $X$. Moreover, if $T$ is a bijective map, then
\begin{equation}\label{invertible}
\gamma(T)=\m(T)=\q(T)=\M(T)=\|T^{-1}\|^{-1}.
\end{equation}
Note that it follows from the above definitions that the spectral functions $\m(.)$ and $\q(.)$ are contractive, and thus $\M(.)$ is a continuous function. But, unlike these, the spectral function $\gamma(.)$ is not continuous
as the following simple example shows:
$$\gamma\left(\left[
\begin{array}{cc}
1&0\\
0&1/n
\end{array}
\right]\right)=1/n\to0\not=1=\gamma\left(\left[
\begin{array}{cc}
1&0\\
0&0
\end{array}
\right]\right)\text{ as }n\to\infty.$$

Several spectra can be described in term of the above spectral quantities. The \emph{generalized spectrum} of an operator $T\in\lx$ is 
$\sigma_g(T):=\{\lambda\in\CC:\lim_{z\to\lambda}\gamma(T-z)=0\},$ and the surjectivity spectrum and the approximate point spectrum of $T$ are given by $\sigma_{su}(T):=\{\lambda\in\CC:\q(T-\lambda)=0\}$ and $\sigma_{ap}(T):=\{\lambda\in\CC:\m(T-\lambda)=0\}.$ 
All these are closed subsets of  $\sigma(T)$, the spectrum of $T$, and contain the boundary of $\sigma(T)$; see for instance \cite{Mu}.  In particular, the spectral radius, $\r(T)$, of $T$ coincides with the maximum modulus of each of the previous mentioned spectra. Thus, applying \cite[Theorem 3.2]{ZFZH} or \cite[Theorem 1]{BreSem96}, one immediately gets 
the complete description of additive or linear surjective maps $\varphi$ on $\lx$ preserving any one of the above spectra.

Now, we are ready to state and prove a more general result than the promised one. Its proof depends on some arguments quoted from \cite[Proof of Theorem 7.2]{BBS}. Given $x\in X$ and $f\in X^*$,
we write $\langle x,f\rangle$ instead of $f(x)$ and $x\otimes f$ for the rank one  operator defined by $x\otimes f(y):=\langle y,f\rangle x,~(y\in X)$.

\begin{theorem}\label{ineq-gamma}
For an additive  surjective map $\varphi: \lx \to \lx$, there are $\alpha,~\beta>0$ such that $\beta\gamma(T)\leq \gamma(\varphi(T))\leq\alpha\gamma(T)$ for all $T\in\lx$ if and only if 
either there are bijective continuous mappings $A:X\to X$ and $B:X\to X$ both linear or both conjugate linear such that $\varphi(T)=ATB$ for all $T\in\lx$, or  
there are bijective continuous mappings $A:X^*\to X$ and $B:X\to X^*$ both linear or both conjugate linear such that $\varphi(T)=AT^*B$ for all $T\in\lx$. The last case may occur only if $X$ is reflexive.
\end{theorem}
\begin{proof}
Obviously, we only need to prove the "only if" part. Assume that there are $\alpha,~\beta>0$ such that
\begin{equation}\label{gamma}
\beta\gamma(T)\leq \gamma(\varphi(T))\leq\alpha\gamma(T)
\end{equation} for all $T\in\lx$, and let us prove that $\varphi$ preserves the zeros of $\M(.)$ in both directions (i.e., if $T\in\lx$, then  $\M(T)=0\iff\M(\varphi(T))=0$).
We first show that $\varphi$ is injective. Assume that $\varphi(T_0)=0$ for some $T_0\in\lx$, and note that it follows from (\ref{gamma}) that
$${\beta}/{\alpha}\gamma(T+T_0-\lambda)\leq\gamma(T-\lambda)\leq{\alpha}/{\beta}\gamma(T+T_0-\lambda)$$
for all $T\in\lx$ and all $\lambda\in\CC$. It follows that 
\begin{eqnarray*}
\lambda\in\sigma_g(T+T_0)&\iff&\lim_{z\to\lambda}\gamma(T+T_0-\lambda)=0\\
&\iff&\lim_{z\to\lambda}\gamma(T-\lambda)=0\\
&\iff&\lambda\in\sigma_g(T),
\end{eqnarray*}
and $\sigma_g(T+T_0)=\sigma_g(T)$ for all $T\in\lx$. As the boundary of the  spectrum is contained in the generalized spectrum, we have  $\r(T+T_0)=\r(T)$ for all $T\in\lx$ and $T_0=0$ by
Zem\' anek spectral characterization of the radical; see \cite[Theorem 5.3.1]{Au}. Therefore, $\varphi$ is injective and is, in fact, a bijective map and its inverse satisfies
a similar inequality to (\ref{gamma}). So, we only need to show that $\varphi$ preserves the zeros of $\M(.)$ in one direction.

Now, assume that $T_0\in\lx$ is an operator for which $\M(T_0)>0$ and let us show that $\M(R_0)>0$ where $R_0:=\varphi(T_0)$. 
Note that, since $\gamma(R_0)\geq\beta\gamma(T_0)>0$, the operator $R_0$ has a closed range. To see that $\M(R_0)>0$, it suffices  to show that 
$R_0$ is surjective or $\ker(R_0)$ is trivial.
Assume by the way of contradiction that $R_0$ is not surjective and $\ker(R_0)$ is not trivial, and pick 
up two unit vectors $x\not\in\ran(R_0)$ and $y\in\ker(R_0)$. Let $f\in X^*$ be a linear functional such that $\langle y,f\rangle=1$, and 
$r>0$ be a positive rational number. Since $x\not\in\ran(R_0)$, we have $\ker(R_0+rx\otimes f)=\ker(R_0)\cap\ker(f)$ and
\begin{eqnarray*}
r&=&\|(R_0+rx\otimes f)y\|\\
&\geq&\gamma(R_0+rx\otimes f)\dist\big(y,\ker(R_0+rx\otimes f)\big)\\
&\geq&\delta\gamma(R_0+rx\otimes f),
\end{eqnarray*}
where $\delta:=\dist\big(y,\ker(f)\big)$ which is of course positive.
Since $\varphi$ is surjective, there is $S_0\in\lx$ such that $\varphi(S_0)=x\otimes f$.  Keep in mind that $\varphi$ is 
$\QQ$-linear and note that it follows from (\ref{gamma}) that 
$$r\geq\delta\gamma(R_0+rx\otimes f)=\delta\gamma(\varphi(T_0+rS_0)\geq\beta \delta\gamma(T_0+rS_0).$$
As the set of all operators with positive maximum modulus is open, it follows that $\M(T_0+r S_0)>0$ and  thus 
$$r\geq\beta \delta\gamma(T_0+rS_0)=\beta\delta\M(T_0+r S_0)$$
for all sufficiently small rational numbers $r$. As the maximum modulus is a continuous function, the right side of the inequality tends to $\beta\delta M(T_0)>0$ as $r$ goes to $0$,
and thus one gets a contradiction. We therefore have $\M(R_0)=\M(\varphi(T_0))>0$;  as desired. 

Finally, apply next lemma to get the desired forms of $\varphi$. 
\end{proof}

The next lemma and its proof were sitting in  \cite[Theorem 3.1 and its proof]{Cui-Hou} and needed only a simple step to be discovered therein.  It was also observed in \cite{bendaoud} but only 
in the Hilbert space operators case; see \cite[Corollary 2.3]{bendaoud}. 

\begin{lemma}\label{c}Assume that $\c(.)$ stands for any one of the spectral quantities $\m(.)$, $\q(.)$ and $\M(.)$. If $\varphi:\lx\to\ly$ is an additive surjective map
preserving the zeros of $\c(.)$ (i.e., if $T\in\lx$, then  $\c(T)=0\iff\c(\varphi(T))=0$),  then either there are bijective continuous mappings $A:X\to Y$ and $B:Y\to X$ both linear or both conjugate linear such that $\varphi(T)=ATB$ for all $T\in\lx$, or there are bijective continuous mappings $A:X^*\to Y$ and $B:Y\to X^*$ both linear or both conjugate linear such that $\varphi(T)=AT^*B$ for all $T\in\lx$. This case may occur only if $X$ and $Y$ are reflexive.
\end{lemma}
\begin{proof}We only need to show that $\varphi(\uno)$ is invertible and apply \cite[Theorem 3.3]{Cui-Hou} to $\Phi:={\varphi(\uno)}^{-1}\varphi$ to get the desired conclusion.

Just as at the beginning of the proof of the previous theorem, one can show that $\varphi$ is injective. 
So, the map $\varphi$ is, in fact, bijective  and its inverse $\varphi^{-1}$ preserves  the zeros of $\c(.)$ as well. Lemma 2.1 of  \cite{Cui-Hou} applied to  $\varphi$ and its inverse shows that $\varphi$ is an 
additive bijection between the ideals ${\mathcal F}(X)$  and  ${\mathcal F}(Y)$ of all finite rank operators on $X$ and $Y$, and that $\varphi$ preserves rank-one operators in both directions. The complete description
of such map $\varphi$ when restricted to ${{\mathcal F}(X)}$, given by \cite[Theorem 3.3]{OS}, guaranties  that for any given nonzero element $g\in Y^*$ (resp. $y\in Y$), there are $x\in X$, $f\in X^*$, and $y\in Y$ (resp. $g\in Y^*$)
such that $\langle x,f\rangle=1$ and $\varphi(x\otimes f)=y\otimes g$, and thus
$$\varphi(\uno-x\otimes f)=\varphi(\uno)-\varphi(x\otimes f)=\varphi(\uno)-y\otimes g.$$
Note that, since $\c(\varphi(\uno))>0$, the range $\ran(\varphi(\uno))$ of $\varphi(\uno)$ is closed and so are  $\ran(\varphi(\uno)-y\otimes g)$  and $\ran({\varphi(\uno)}^*-g\otimes y)$. 
As $\c(\varphi(\uno)-y\otimes g)=\c(\uno-x\otimes f)=0$, it follows that $\varphi(\uno)-y\otimes g$ is not injective in case $\c(.)=\m(.)$, and  ${\varphi(\uno)}^*-g\otimes y$ is not injective in case $\c(.)=\q(.)$.
Of course these two operators are not injective in case $\c(.)$ coincides with $\M(.)$. So, to finish the proof of this lemma, we shall discuss three cases.

{\bf Case 1.} Assume that $\c(.)=\m(.)$, and note that $\varphi(\uno)$ is injective as well. Therefore, we only need to show that $\varphi(\uno)$ is surjective. 
Take an arbitrary nonzero element $y\in Y$, and note that, by what has been discussed above, there is $g\in Y^*$ such that $\varphi(\uno)-y\otimes g$ is not 
injective and $(\varphi(\uno)-y\otimes g)z=0$ for some nonzero element $z\in Y$. The injectivity of $\varphi(\uno)$ ensures that $g(z)\not=0$ and implies that $y$ lies in $\ran(\varphi(\uno))$.
This shows that $\varphi$ is surjective and implies that $\varphi(\uno)$ is invertible; as desired.

{\bf Case 2.} Assume that $\c(.)=\q(.)$ and note that ${\varphi(\uno)}^*$ is injective. Pick up an arbitrary nonzero element $g\in Y^*$, and note that
 ${(\varphi(\uno)}^*-g\otimes y)h=0$ for some $y\in Y$ and $0\not=h\in Y^*$. Just as above, we see that $\langle y,h\rangle\not=0$ and $g$ lies in the range of ${\varphi(\uno)}^*$.
 This implies that ${\varphi(\uno)}^*$ is surjective and that $\varphi(\uno)$ is invertible in this case too; as desired.

{\bf Case 3.} Assume finally that $\c(.)=\M(.)$ and note that either $\varphi(\uno)$ is injective or ${\varphi(\uno)}^*$ is injective. If $\varphi(\uno)$ is injective, then, just as in Case 1, we see that $\varphi(\uno)$ is invertible.
When ${\varphi(\uno)}^*$ is injective, then, just as in Case 2, we see that $\varphi(\uno)$ is invertible in this case too. 
\end{proof}

The promised result describes additive surjective maps preserving the reduced minimum modulus of Banach space operators. It extends \cite[Theorem 7.1]{BBS} to the additive preservers and Banach space operators setting, and shows that the 
condition that  $\varphi(\uno)$ is invertible in \cite[Theorem 4.2]{skhiri} is superfluous.

\begin{theorem}\label{gamma-preserving}An additive surjective map $\varphi:\lx\to\ly$ preserves the reduced minimum modulus (i. e., $\gamma(\varphi(T))=\gamma(T)$ for all $T\in\lx$) if and only if either
there are bijective isometries $U:X\to Y$ and $V:Y\to X$ both linear or both conjugate linear such that $\varphi(T)=UTV$ for all $T\in\lx$, or  there are bijective isometries $U:X^*\to Y$ and $V:Y\to X^*$ both linear or both conjugate linear such that $\varphi(T)=UT^*V$ for all $T\in\lx$.
The last case can not occur if any one of $X$ and $Y$ is not reflexive. 
\end{theorem}

The proof of Theorem \ref{gamma-preserving} uses the following lemmas  quoted from  \cite[Theorem 3.1 and Corollary 3.2]{skhiri}. 
The proofs presented therein are long and require several computations and applications of Hahn-Banach 
Theorem. Here, we propose simple and shorter proofs.
\begin{lemma}\label{isometry}For a bijective mapping  $A\in\lxy$, the following statements are equivalent.
\begin{itemize}
\item[(i)] $\|ATA^{-1}\|=\|T\|$ for all invertible operators $T\in\lx$.
\item[(ii)] $\|ATA^{-1}\|\leq\|T\|$ for all invertible operators $T\in\lx$.
\item[(iii)] $\|ATA^{-1}\|\geq\|T\|$ for all invertible operators $T\in\lx$.
\item[(iv)] $A$ is an isometry multiplied by a scalar. 
\end{itemize}
\end{lemma}
\begin{proof}Obviously, the implications  (i)$\Rightarrow$(ii), (i)$\Rightarrow$(iii) and (iv)$\Rightarrow$(i) are always there.
So, we only need to establish the implications (ii)$\Rightarrow$(iv) and (iii)$\Rightarrow$(iv). 

Assume that $\|ATA^{-1}\|\leq\|T\|$ for all invertible operators $T\in\lx$. In particular, we have
$\|\frac{\uno}{n}+A(x\otimes f)A^{-1}\|=\|A\left(\frac{\uno}{n}+x\otimes f\right)A^{-1}\|\leq\|\frac{\uno}{n}+x\otimes f\|$
for all positive integers $n$, $x\in X$ and $f\in X^*$. Taking the limit, as $n$ goes to $\infty$, of both sides of this inequality, we get that
$\|Ax\|\|{A^{-1}}^*f\|=\|A(x\otimes f)A^{-1}\|\leq \|x\otimes f\|=\|x\|\|f\|$
for all $x\in X$ and $f\in X^*$. Thus, $\|A\|\|A^{-1}\|\leq1$ and 
$\|A\|\|x\|\leq\frac{\|x\|}{\|A^{-1}\|}\leq\|Ax\|\leq\|A\|\|x\|$
for all $x\in X$. This shows that ${A}/{\|A\|}$ is a bijective isometry and establishes the implication (ii)$\Rightarrow$(iv).

Now, assume that $\|ATA^{-1}\|\geq\|T\|$ for all invertible operators $T\in\lx$. It follows that  $\|A^{-1}SA\|\leq\|S\|$ for all invertible operators $S\in\ly$
and ${A^{-1}}/{\|A^{-1}\|}$ is a bijective isometry by the established  implication (ii)$\Rightarrow$(iv). Hence ${A}/{\|A\|}$ is a bijective isometry as well, and the implication (iii)$\Rightarrow$(iv)
is established.
\end{proof}
\begin{lemma}\label{AB}For two bijective transformations $A\in\lxy$ and $B\in\lyx$, the following statements are equivalent.
\begin{itemize}
\item[(i)] $\|ATB\|=\|T\|$ for all invertible operators $T\in\lx$.
\item[(ii)] $A$ and $B$ are isometries multiplied by scalars $\lambda$ and $\mu$ such that $|\lambda\mu|=1$.
\end{itemize}
\end{lemma}
\begin{proof}
We only need to show that the first statement implies the other one. So, assume that $\|ATB\|=\|T\|$ for all invertible operators $T\in\lx$, and note that 
$\|AB\|=\|A^{-1}B^{-1}\|=1.$ Thus for every invertible operator $T\in\lx$, we have
\[
\|ATA^{-1}\|=\|A(TA^{-1}B^{-1})B\|=\|TA^{-1}B^{-1}\|\leq\|T\|\|A^{-1}B^{-1}\|=\|T\|.
\]
By Lemma \ref{isometry}, there is an isometry $U$ and a scalar $\lambda$ such that $A=\lambda U$.  By similar argument, we also see that $B=\mu V$ for some  isometry $V$ and a scalar $\mu$. These together with the fact that $\|AB\|=1$ imply that $|\lambda\mu|=1$, and the proof is complete. 
\end{proof}

We are now in a position to prove Theorem \ref{gamma-preserving}.
\begin{proof}[Proof of Theorem \ref{gamma-preserving}]
Suppose that  $\varphi:\lx\to\lx$ is an additive surjective map for which $\gamma(\varphi(T))=\gamma(T)$ for all $T\in\lx$.  By Theorem \ref{ineq-gamma}, either 
there are bijective continuous mappings $A:X\to X$ and $B:X\to X$ both linear or both conjugate linear such that $\varphi(T)=ATB$ for all $T\in\lx$, or  
there are bijective continuous mappings $A:X^*\to X$ and $B:X\to X^*$ both linear or both conjugate linear such that $\varphi(T)=AT^*B$ for all $T\in\lx$.

Assume without loss of generality that the first possibility holds, and note that 
$$\frac{1}{\|T\|}={\gamma(T^{-1})}={\gamma(\varphi(T^{-1}))}=\frac{1}{\|\varphi(T^{-1})^{-1}\|}=\frac{1}{\|B^{-1}TA^{-1}\|}$$
for all invertible operators $T\in\lx$. By Lemma \ref{AB}, there are isometries $U:X\to Y$ and $V:Y\to X$ both linear or both conjugate linear, and scalars $\lambda$ and $\mu$ such that
$A=\lambda U$ and $B=\mu V$ and $\lambda \mu=1$. Thus
 $$\varphi(T)=ATB=(\lambda U)T(\mu V)=UTV$$ for all $T\in\lx$; as desired. 
\end{proof}

Before closing this section, we mention that the statement of Theorem \ref{ineq-gamma} and Theorem \ref{gamma-preserving} for the Hilbert space operators case
need be slightly modified in an obvious way, and that, in view of (\ref{invertible}), Lemma \ref{isometry} and Lemma \ref{AB} can be stated in a similar way when replacing 
the norm by any one of the spectral functions $\gamma(.)$, $\m(.)$, $\q(.)$, and $\M(.)$.

\section{Consequences and comments}
This section is devoted for some comments and applications of Lemma \ref{c} and Lemma \ref{AB}. 
Having these lemmas in hand, the same proof of Theorem \ref{gamma-preserving}, with no extra efforts,  yields the following two theorems. The first one describes surjective additive maps from $\lx$ onto $\ly$ preserving  the minimum and the surjectivity moduli
of Banach space operators. While the other one characterizes surjective additive maps from $\lx$ onto $\ly$ preserving  the maximum modulus.

\begin{theorem}\label{ms}If $\varphi:\lx\to\ly$ is an additive surjective map preserving either the minimum modulus or the surjectivity modulus, then either
there are bijective isometries $U:X\to Y$ and $V:Y\to X$ both linear or both conjugate linear such that $\varphi(T)=UTV$ for all $T\in\lx$, or 
 there are bijective isometries $U:X^*\to Y$ and $V:Y\to X^*$ both linear or both conjugate linear such that $\varphi(T)=UT^*V$ for all $T\in\lx$.
\end{theorem}
From the definitions of the minimum and surjectivity moduli, these quantities are always preserved by maps of the form $\varphi(T)=UTV,\,(T\in\lx),$ 
where $U:X\to Y$ and $V:Y\to X$ are both linear or both conjugate linear bijective isometries.  While if $\varphi$ preserves the minimum modulus (resp. the surjectivity modulus), 
then the second conclusion of the previous theorem can not occur if any one of $X$ and $Y$ is not reflexive or if there is a non invertible surjective (resp. non invertible bounded below) operator in $\lx$.
We also mention that in \cite{GM},  Gowers and  Maurey constructed an infinite-dimensional, separable, reflexive complex Banach space $X$
such that $\sigma(T)$ is countable for all $T\in\lx$. Therefore, $\sigma(T)=\sigma_{ap}(T)=\sigma_{su}(T)$ for all $T\in\lx$, and every surjective or bounded below linear operator in $\lx$ is invertible.

\begin{theorem}\label{mm}An additive surjective map $\varphi:\lx\to\ly$ preserves the maximum modulus (i. e., $\M(\varphi(T))=\M(T)$ for all $T\in\lx$) if and only if either
there are bijective isometries $U:X\to Y$ and $V:Y\to X$ both linear or both conjugate linear such that $\varphi(T)=UTV$ for all $T\in\lx$, or  there are bijective isometries $U:X^*\to Y$ and $V:Y\to X^*$ both linear or both conjugate linear such that $\varphi(T)=UT^*V$ for all $T\in\lx$.
The last case can not occur if any one of $X$ and $Y$ is not reflexive. 
\end{theorem}

In \cite{Mb3, skhiri}, surjective linear maps on $\lx$ preserving and compressing the generalized spectrum are characterized.  Inspecting the proof of the main result of \cite{hou}, a little bit more can be obtained.

\begin{theorem}
For an additive surjective map $\varphi: \lx \to \ly$, the following are equivalent.
\begin{itemize}
\item[(i)] $\varphi$ preserves the generalized spectrum (i.e., $\sigma_g(\varphi(T))=\sigma_g(T)$ for all $T\in\lx$).
\item[(ii)] $\varphi$ does not annihilate all rank-one idempotents, and compresses the generalized spectrum (i.e., $\sigma_g(\varphi(T))\subset\sigma_g(T)$ for all $T\in\lx$).
\item[(iii)] $\varphi$ decompresses the generalized spectrum (i.e., $\sigma_g(\varphi(T))\supset\sigma_g(T)$ for all $T\in\lx$).
\item[(iv)]Either $\varphi(T)=ATA^{-1},~(T\in\lx),$ for some isomorphism $A\in\lxy$, or  $\varphi(T)=BT^*B^{-1},~(T\in\lx),$ for some isomorphism  $B\in\mathcal{B}(X^*,Y)$. The last case may occur only if $X$ and $Y$ are reflexive.
\end{itemize}
\end{theorem}

We close this paper with a remark. Assume that $\c(.)$ stands for any one of the spectral quantities $\m(.)$, $\q(.)$, $\M(.)$ and $\gamma(.)$, and let $\varphi$ 
be a surjective linear map on $\lx$. Having the paper \cite{BBS} in hand, one can see that,  only if a mild condition on $\varphi(\uno)$ is imposed, the conclusions of above 
results remain the same if replacing the hypothesis ``$\varphi$ preserves the spectral quantity $\c(.)$" by  ``$\varphi$ satisfies either $\c(\varphi(T))\leq\c(T)$ for all $T\in\lx$ or $\c(\varphi(T))\geq\c(T)$ for all $T\in\lx$". For further details, we refer the reader to \cite{BBS}.

\end{document}